\documentclass[12pt]{article}

\usepackage{amsmath, epsfig, cite}
\usepackage{amssymb}
\usepackage{amsfonts}
\usepackage{latexsym}
\usepackage{amsthm}

\newtheorem{thm}{Theorem}[section]

\newtheorem{lem}[thm]{Lemma}

\numberwithin{equation}{section}

\renewcommand{\thefootnote}

\begin{document}

\begin{center}
{\large\bf Some fast convergent series for  the mathematical\\
constants $\zeta(4)$ and $\zeta(5)$
 \footnote{ The work is supported by the National Natural Science Foundation of China (No. 12071103).}}
\end{center}

\renewcommand{\thefootnote}{$\dagger$}

\vskip 2mm \centerline{Chuanan Wei}
\begin{center}
{School of Biomedical Information and Engineering\\ Hainan Medical
University, Haikou 571199, China
\\
 Email address: weichuanan78@163.com}
\end{center}


\vskip 0.7cm \noindent{\bf Abstract.} Recently, Sun [preprint,
arXiv: 2210.07238v7] proposed two conjectural series for the
mathematical constant $\zeta(4)$ and two conjectural series for the
mathematical constant $\zeta(5)$. In terms of the operator method
and two hypergeometric transformations, we
 prove these four conjectures. Furthermore, we also find
 some new series for the two
 constants in this paper.



\section{Introduction}

 For two complex sequences
$\{a_i\}_{i\geq1}$, $\{b_j\}_{j\geq1}$ and  a complex variable $z$,
define the hypergeometric series to be
$$
_{r+1}F_{r}\left[\begin{array}{c}
a_1,a_2,\ldots,a_{r+1}\\
b_1,b_2,\ldots,b_{r}
\end{array};\, z
\right] =\sum_{k=0}^{\infty}\frac{(a_1)_k(a_2)_k\cdots(a_{r+1})_k}
{(1)_k(b_1)_k\cdots(b_{r})_k}z^k,
$$
where the shifted-factorial has been given by
\begin{align*}
(x)_0=1 \quad\text{and}\quad (x)_n=x(x+1)\cdots(x+n-1) \quad
\text{for} \quad n\in \mathbb{Z}^{+}.
\end{align*}
Then a hypergeometric transformation between two  $_9F_8$ series
(cf.\cite[P. 27]{Bailey}) can be expressed as
\begin{align}
&{_{9}F_{8}}\left[\begin{array}{cccccccc}
  a,1+\frac{a}{2},b,c,d,e,f,g,-n\\
  \frac{a}{2},1+a-b,1+a-c,1+a-d,1+a-e,1+a-f,1+a-g,1+a+n
\end{array};1\right]
\notag\\
&=\frac{(1+a)_n(1+\lambda-e)_n1+\lambda-f)_n1+\lambda-g)_n}{(1+\lambda)_n(1+a-e)_n(1+a-f)_n(1+a-g)_n}
\notag\\
&\times{_{9}F_{8}}\left[\begin{array}{cccccccc}
  \lambda,1+\frac{\lambda}{2},\lambda+b-a,\lambda+c-a,\lambda+d-a,e,f,g,-n\\
  \frac{\lambda}{2},1+a-b,1+a-c,1+a-d,1+\lambda-e,1+\lambda-f,1+\lambda-g,1+\lambda+n\end{array};1
\right], \label{9F8}
\end{align}
where $\lambda=1+2a-b-c-d$ and $2+3a=b+c+d+e+f+g-n$.

For $\ell,n\in \mathbb{Z}^{+}$, define the generalized $\ell$-order
harmonic numbers by
\[H_{n}^{(\ell)}(x)=\sum_{k=1}^n\frac{1}{(x+k)^{\ell}},\]
where  $x$ is a complex variable. The $x=0$ case of them are the
$\ell$-order harmonic numbers
\[H_{n}^{(\ell)}=\sum_{k=1}^n\frac{1}{k^{\ell}}.\]
For $s\in C$ with $\mathfrak{R}(s)>1$, define the Riemann zeta
function as
$$\zeta(s)=\sum_{k=0}^{\infty}\frac{1}{k^s}.$$
Three nice series for $\zeta(3)$ read
\begin{align}
&\sum_{k=0}^{\infty}\bigg(\frac{-1}{1024}\bigg)^k\frac{(1)_k^5}{(\frac{3}{2})_k^5}(77+250k+205k^2)=64\zeta(3),
\label{zeta-ta}\\[1mm]
&\quad\sum_{k=0}^{\infty}\bigg(\frac{-1}{4}\bigg)^k\frac{(1)_k^5}{(\frac{3}{2})_k^5}(5+14k+10k^2)=\frac{7}{2}\zeta(3),
\label{zeta-tb}\\[1mm]
&\quad\:\sum_{k=0}^{\infty}\bigg(\frac{1}{16}\bigg)^k\frac{(1)_k^2}{(\frac{3}{2})_k^2}
\frac{19+30k}{(1+k)(1+2k)}=16\zeta(3). \label{zeta-td}
\end{align}
where \eqref{zeta-ta} is due to Amdeberhan and Zeilberger
\cite{Amdeberhan-b}, \eqref{zeta-tb} owns to Guillera
\cite{Guillera}, and \eqref{zeta-td} can be viewed in Chu and Zhang
\cite{Chu-b}. It is well known that
\begin{align*}
\zeta(4)=\sum_{k=0}^{\infty}\frac{1}{k^4}=\frac{\pi^4}{90}.
\end{align*}
In the development history of mathematics, series for $\zeta(4)$ and
$\zeta(5)$ are rare. We state two formulae as follows:
\begin{align}
&\qquad\qquad\qquad\quad
\sum_{k=1}^{\infty}\frac{H_k}{k^3}=\frac{5}{4}\zeta(4),
\label{zeta4-a}\\[1mm]
 &\sum_{k=0}^{\infty}\bigg(\frac{-1}{4}\bigg)^k\frac{(1)_k}{(\frac{3}{2})_k}\bigg\{\frac{4}{(1+k)^4}-\frac{5H_k^{(2)}}{(1+k)^2}\bigg\}=4\zeta(5),
\label{zeta4-b}
\end{align}
where \eqref{zeta4-a} first appeared in a 1742 letter by Goldbach to
Euler (cf. \cite{Weisstein}) and \eqref{zeta4-b} is from Borwein and
Bradley \cite{Borwein}. For more results on mathematical constants,
the reader is referred to
 the papers \cite{Au,Schlosser,Chan,Guo,Wang}.

Motivated by \eqref{zeta-ta} and \eqref{zeta-tb}, Sun
\cite[Equations (4.1), (4.3), (4.6), and (4.8)]{Sun} proposed the
following four surprising conjectures.

\begin{thm}\label{thm-a}
\begin{align}
\sum_{k=0}^{\infty}\bigg(\frac{-1}{1024}\bigg)^k\frac{(1)_k^5}{(\frac{3}{2})_k^5}
\Big\{(77+250k+205k^2)\Big(H_{1+2k}-H_{k}\Big)-41k-25\Big\}=48\zeta(4).
 \label{eq:wei-a}
\end{align}
\end{thm}

\begin{thm}\label{thm-b}
\begin{align}
\sum_{k=0}^{\infty}\bigg(\frac{-1}{1024}\bigg)^k\frac{(1)_k^5}{(\frac{3}{2})_k^5}
\Big\{(77+250k+205k^2)\Big[4H_{1+2k}^{(2)}-12H_{k}^{(2)}\Big]-43\Big\}=256\zeta(5).
 \label{eq:wei-b}
\end{align}
\end{thm}

\begin{thm}\label{thm-c}
\begin{align}
\sum_{k=0}^{\infty}\bigg(\frac{-1}{4}\bigg)^k\frac{(1)_k^5}{(\frac{3}{2})_k^5}
\Big\{(5+14k+10k^2)\Big(2H_{1+2k}-H_{k}\Big)-3k-2\Big\}=\frac{45}{8}\zeta(4).
 \label{eq:wei-c}
\end{align}
\end{thm}

\begin{thm}\label{thm-d}
\begin{align}
\sum_{k=0}^{\infty}\bigg(\frac{-1}{4}\bigg)^k\frac{(1)_k^5}{(\frac{3}{2})_k^5}
\Big\{(5+14k+10k^2)\Big[4H_{1+2k}^{(2)}-3H_{k}^{(2)}\Big]-2\Big\}=\frac{31}{2}\zeta(5).
 \label{eq:wei-d}
\end{align}
\end{thm}

Encouraged by \eqref{zeta-td}, we shall establish the following four
theorems.

\begin{thm}\label{thm-e}
\begin{align}
\sum_{k=0}^{\infty}\bigg(\frac{1}{16}\bigg)^k\frac{(1)_k^2}{(\frac{3}{2})_k^2}
\bigg\{\frac{19+30k}{(1+k)(1+2k)}H_{1+2k}+\frac{13}{2(1+k)^2}\bigg\}=24\zeta(4).
 \label{eq:wei-e}
\end{align}
\end{thm}

\begin{thm}\label{thm-f}
\begin{align}
\sum_{k=0}^{\infty}\bigg(\frac{1}{16}\bigg)^k\frac{(1)_k^2}{(\frac{3}{2})_k^2}
\bigg\{\frac{19+30k}{(1+k)(1+2k)}(15H_{1+2k}-26H_{k}\Big)+\frac{104}{(1+2k)^2}\bigg\}=360\zeta(4).
 \label{eq:wei-f}
\end{align}
\end{thm}

\begin{thm}\label{thm-g}
\begin{align}
\sum_{k=0}^{\infty}\bigg(\frac{1}{16}\bigg)^k\frac{(1)_k^2}{(\frac{3}{2})_k^2}
\bigg\{\frac{19+30k}{(1+k)(1+2k)}H_{k}^{(2)}-\frac{17}{2(1+k)^3}\bigg\}
=-8\zeta(5).
 \label{eq:wei-g}
\end{align}
\end{thm}

\begin{thm}\label{thm-h}
\begin{align}
\sum_{k=0}^{\infty}\bigg(\frac{1}{16}\bigg)^k\frac{(1)_k^2}{(\frac{3}{2})_k^2}
\bigg\{\frac{19+30k}{(1+k)(1+2k)}\Big[H_{1+2k}^{(2)}-2H_{k}^{(2)}\Big]-\frac{9}{4(1+k)^3}\bigg\}
=16\zeta(5)
 \label{eq:wei-h}
\end{align}
\end{thm}

We point out that the combination of  \eqref{eq:wei-g} and
\eqref{eq:wei-h} can provide the following interesting conclusion:
\begin{align*}
\sum_{k=0}^{\infty}\bigg(\frac{1}{16}\bigg)^k\frac{(1)_k^2}{(\frac{3}{2})_k^2}\frac{19+30k}{(1+k)(1+2k)}
\Big\{34H_{1+2k}^{(2)}-77H_{k}^{(2)}\Big\}=616\zeta(5).
\end{align*}

Via the limit of the difference operator, we may define the
derivative operator $\mathcal{D}_{x}$ by
\begin{align*}
\mathcal{D}_{x}f(x)=\lim_{\Delta x\to0}\frac{f(x+\Delta
x)-f(x)}{\Delta x}.
\end{align*}
Thus there is  the relation:
\begin{align*}
&\mathcal{D}_{x}(x)_n=(x)_nH_n(x-1).
\end{align*}

The structure of the paper is organized as follows. Through the
derivative operator and two hypergeometric transformations, we shall
verify Theorem \ref{thm-a}-\ref{thm-d} in Section 2. Similarly, the
proof of Theorems \ref{thm-e}-\ref{thm-h} will be displayed in
Section 3.

\section{Proof of Theorems \ref{thm-a}-\ref{thm-d}}
For the aim to prove Theorem \ref{thm-a}, we call for the following
lemma.

\begin{lem}\label{lemm-a}
The series
\begin{align*}
g(b)=\sum_{k=0}^{\infty}\frac{1}{(1+k)^2}\frac{(b)_k}{(3-b)_{k}}
\end{align*}
is uniformly convergent in the open interval
$(\frac{1}{2},\frac{3}{2})$.
\end{lem}

\begin{proof}
For any $b\in(\frac{1}{2},\frac{3}{2})$, we have
\begin{align*}
g(b)\leq\sum_{k=0}^{\infty}\frac{1}{(1+k)^2}\frac{(\frac{3}{2})_k}{(3-\frac{3}{2})_{k}}
=\sum_{k=0}^{\infty}\frac{1}{(1+k)^2}.
\end{align*}
Noting that the right series is convergent, we conclude that $g(b)$
is uniformly convergent in the interval $(\frac{1}{2},\frac{3}{2})$
thanks to Weierstrass's M-test.
\end{proof}

\begin{proof}[{\bf{Proof of Theorem \ref{thm-a}}}]
Remember a hypergeometric transformation (cf. \cite[Theorem
31]{Chu-b}):
\begin{align}
&\sum_{k=0}^{\infty}(-1)^k\frac{(b)_k(c)_k(d)_k(e)_k(1+a-b-c)_k(1+a-b-d)_{k}(1+a-b-e)_{k}}{(1+a-b)_{2k}(1+a-c)_{2k}(1+a-d)_{2k}(1+a-e)_{2k}}
\notag\\[1mm]
&\quad\times\frac{(1+a-c-d)_k(1+a-c-e)_{k}(1+a-d-e)_{k}}{(1+2a-b-c-d-e)_{2k}}\alpha_k(a,b,c,d,e)
\notag\\[1mm]
&\:=\sum_{k=0}^{\infty}(a+2k)\frac{(b)_k(c)_k(d)_k(e)_k}{(1+a-b)_{k}(1+a-c)_{k}(1+a-d)_{k}(1+a-e)_{k}},
\label{equation-aa}
\end{align}
where $\mathfrak{R}(1+2a-b-c-d-e)>0$ and

\begin{align*}
&\alpha_k(a,b,c,d,e)\\[1mm]
&\:=\frac{(1+2a-b-c-d+3k)(a-e+2k)}{1+2a-b-c-d-e+2k}+\frac{(e+k)(1+a-b-c+k)}{(1+a-b+2k)(1+a-d+2k)}
\\[1mm]
&\quad\times\frac{(1+a-b-d+k)(1+a-c-d+k)(2+2a-b-d-e+3k)}{(1+2a-b-c-d-e+2k)(2+2a-b-c-d-e+2k)}
\\[1mm]
&\:+\frac{(c+k)(e+k)(1+a-b-c+k)(1+a-b-d+k)}{(1+a-b+2k)(1+a-c+2k)(1+a-d+2k)(1+a-e+2k)}
\\[1mm]
&\quad\times\frac{(1+a-b-e+k)(1+a-c-d+k)(1+a-d-e+k)}{(1+2a-b-c-d-e+2k)(2+2a-b-c-d-e+2k)}.
\end{align*}

Choose $(a,c,d,e)=(2,1,1,1)$ in \eqref{equation-aa} to obtain
\begin{align}
\sum_{k=0}^{\infty}\bigg(\frac{-1}{64}\bigg)^k\frac{(b)_k(2-b)_k^3(1)_k^3}{(2-b)_{2k}(3-b)_{2k}(\frac{3}{2})_{k}^3}A_k(b)
=\sum_{k=0}^{\infty}\frac{2}{(1+k)^2}\frac{(b)_k}{(3-b)_{k}},
\label{eq:wei-aa}
\end{align}
where
\begin{align*}
A_k(b)
&=\frac{(1+2k)(3-b+3k)}{2-b+2k}+\frac{(1+k)(4-b+3k)(2-b+k)^2}{2(2-b+2k)(3-b+2k)^2}
\\[1mm]
&\quad+\frac{(1+k)(2-b+k)^3}{8(2-b+2k)(3-b+2k)^2}.
\end{align*}
By means of Lemma \ref{lemm-a}, we learn that the right series
 is uniformly convergent in $(\frac{1}{2},\frac{3}{2})$.
It is obvious that the left series is uniformly convergent in the
same interval. Apply the operator $\mathcal{D}_{b}$ on both sides of
\eqref{eq:wei-aa} to get
\begin{align*}
&\sum_{k=0}^{\infty}\bigg(\frac{-1}{64}\bigg)^k\frac{(b)_k(2-b)_k^3(1)_k^3}{(2-b)_{2k}(3-b)_{2k}(\frac{3}{2})_{k}^3}A_k(b)
\notag\\[1mm]
&\quad\times\Big\{H_{k}(b-1)-3H_{k}(1-b)+H_{2k}(1-b)+H_{2k}(2-b)\Big\}
\notag\\[1mm]
&\:+\sum_{k=0}^{\infty}\bigg(\frac{-1}{64}\bigg)^k\frac{(b)_k(2-b)_k^3(1)_k^3}{(2-b)_{2k}(3-b)_{2k}(\frac{3}{2})_{k}^3}\mathcal{D}_{b}A_k(b)
\notag\\[1mm]
&\:\:= \sum_{k=0}^{\infty}\frac{2}{(1+k)^2}\frac{(b)_k}{(3-b)_{k}}
\Big\{H_{k}(b-1)+H_{k}(2-b)\Big\}.
\end{align*}
The $b\to1$ case of it becomes
\begin{align}
&\sum_{k=0}^{\infty}\bigg(\frac{-1}{1024}\bigg)^k\frac{(1)_k^5}{(\frac{3}{2})_k^5}
\Big\{(77+250k+205k^2)\Big(H_{1+2k}-H_{k}-\tfrac{1}{2}\Big)-41k-25\Big\}
\notag\\[1mm]
&\:=64\sum_{k=0}^{\infty}\frac{H_{1+k}}{(1+k)^3}-32\zeta(4)-32\zeta(3)
=48\zeta(4)-32\zeta(3).
 \label{eq:wei-cc}
\end{align}
where we have used \eqref{zeta4-a} in the last step. Therefore, the
combination of \eqref{zeta-ta} with \eqref{eq:wei-cc} engenders
\eqref{eq:wei-a}.
\end{proof}

 For the purpose of proving Theorem \ref{thm-b}, we ask for the
following lemma.

\begin{lem}\label{lemm-b}
The series
\begin{align*}
g(b)=\sum_{k=0}^{\infty}\frac{1}{1+k}\frac{(b)_k(2-b)_k}{(3-b)_{k}(1+b)_{k}}
\end{align*}
is uniformly convergent in the open interval
$(\frac{1}{2},\frac{3}{2})$.
\end{lem}

\begin{proof}
Let
\begin{align*}
u_k(b)=\frac{1}{1+k},\quad
s_{n}(b)=\sum_{k=0}^n\frac{(b)_k(2-b)_k}{(3-b)_{k}(1+b)_{k}}.
\end{align*}
Considering that the sequence $\{u_k(b)\}_{k\geq0}$ is descending on
$k$ and is uniformly convergent to $0$ in
$(\frac{1}{2},\frac{3}{2})$, it is enough to prove that the sequence
$\{s_n(b)\}_{n\geq0}$ is
  uniformly bounded in $(\frac{1}{2},\frac{3}{2})$ in the light of
Dirichlet's uniform convergence test.

For any $b\in(\frac{1}{2},\frac{3}{2})$, it is ordinary to
understand that
\begin{align*}
s_{n}(b)&=\sum_{k=0}^n\frac{b(2-b)}{(b+k)(2-b+k)}
\leq\frac{9}{4}\sum_{k=0}^n\frac{1}{(\frac{1}{2}+k)^2}\leq9\sum_{k=0}^{\infty}\frac{1}{(1+2k)^2}=\frac{9}{8}\pi^2.
\end{align*}
So we complete the proof of Lemma \ref{lemm-b}.

\end{proof}

\begin{proof}[{\bf{Proof of Theorem \ref{thm-b}}}]
Fix $(a,c,d,e)=(2,2-b,1,1)$ in \eqref{equation-aa} to deduce
\begin{align}
\sum_{k=0}^{\infty}\bigg(\frac{-1}{64}\bigg)^k\frac{(b)_k^3(2-b)_k^3(1)_k}{(1+b)_{2k}(3-b)_{2k}(\frac{1}{2})_{k}(\frac{3}{2})_{k}^2}B_k(b)
=\sum_{k=0}^{\infty}\frac{2}{1+k}\frac{(b)_k(2-b)_k}{(3-b)_{k}(1+b)_{k}},
\label{eq:wei-dd}
\end{align}
where
\begin{align*}
B_k(b) &=(2+3k)+\frac{(b+k)(2-b+k)(4-b+3k)}{4(1+2k)(3-b+2k)}
\\[1mm]
&\quad+\frac{(b+k)(2-b+k)^3}{8(1+2k)(1+b+2k)(3-b+2k)}.
\end{align*}
According to Lemma \ref{lemm-b}, we comprehand that the right series
 is uniformly convergent in
$(\frac{1}{2},\frac{3}{2})$. It is clear that the left series is
uniformly convergent in the same interval. Employ the operator
$\mathcal{D}_{b}$ on both sides of \eqref{eq:wei-dd} to discover
\begin{align}
&\sum_{k=0}^{\infty}\bigg(\frac{-1}{64}\bigg)^k\frac{(b)_k^3(2-b)_k^3(1)_k}{(1+b)_{2k}(3-b)_{2k}(\frac{1}{2})_{k}(\frac{3}{2})_{k}^2}B_k(b)
\notag\\[1mm]\notag
&\quad\times\Big\{3H_{k}(b-1)-3H_{k}(1-b)+H_{2k}(2-b)-H_{2k}(b)\Big\}
\end{align}
\begin{align}
&\:+\sum_{k=0}^{\infty}\bigg(\frac{-1}{64}\bigg)^k\frac{(b)_k^3(2-b)_k^3(1)_k}{(1+b)_{2k}(3-b)_{2k}(\frac{1}{2})_{k}(\frac{3}{2})_{k}^2}\mathcal{D}_{b}B_k(b)
\notag\\[1mm]
&\:\:=
\sum_{k=0}^{\infty}\frac{2}{1+k}\frac{(b)_k(2-b)_k}{(3-b)_{k}(1+b)_{k}}
\Big\{H_{k}(b-1)-H_{k}(1-b)+H_{k}(2-b)-H_{k}(b)\Big\}.
\label{eq:wei-ee}
\end{align}
Dividing both sides by $2-2b$ and turning to the relation
 \begin{align*}
 \frac{1}{v-u-2x}\Big\{H_m(x+u)-H_m(v-x)\Big\}=\sum_{i=1}^m\frac{1}{(x+u+i)(v-x+i)},
\end{align*}
which will usually be utilized without indication, \eqref{eq:wei-ee}
can be manipulated as
\begin{align*}
&\sum_{k=0}^{\infty}\bigg(\frac{-1}{64}\bigg)^k\frac{(b)_k^3(2-b)_k^3(1)_k}{(1+b)_{2k}(3-b)_{2k}(\frac{1}{2})_{k}(\frac{3}{2})_{k}^2}B_k(b)
\notag\\[1mm]
&\quad\times\bigg\{\sum_{i=1}^k\frac{3}{(b-1+i)(1-b+i)}-\sum_{i=1}^{2k}\frac{1}{(b+i)(2-b+i)}\bigg\}
\notag\\[1mm]
&\:+\sum_{k=0}^{\infty}\bigg(\frac{-1}{64}\bigg)^k\frac{(b)_k^3(2-b)_k^3(1)_k}{(1+b)_{2k}(3-b)_{2k}(\frac{1}{2})_{k}(\frac{3}{2})_{k}^2}\frac{\mathcal{D}_{b}B_k(b)}{2-2b}
\notag\\[1mm]
&\:\:=
\sum_{k=0}^{\infty}\frac{2}{1+k}\frac{(b)_k(2-b)_k}{(3-b)_{k}(1+b)_{k}}
\bigg\{\sum_{j=1}^k\frac{1}{(b-1+j)(1-b+j)}-\sum_{j=1}^{k}\frac{1}{(b+j)(2-b+j)}\bigg\}.
\end{align*}
The $b\to1$ case of the last equation is
\begin{align}
&\sum_{k=0}^{\infty}\bigg(\frac{-1}{1024}\bigg)^k\frac{(1)_k^5}{(\frac{3}{2})_k^5}
\Big\{(77+250k+205k^2)\Big[4H_{1+2k}^{(2)}-12H_{k}^{(2)}-4\Big]-43\Big\}
\notag\\[1mm]
&\:=256\zeta(5)-256\zeta(3).
 \label{eq:wei-ff}
\end{align}
Hence, the combination of \eqref{zeta-ta} with \eqref{eq:wei-ff}
produces \eqref{eq:wei-b}.
\end{proof}

So as to prove Theorem \ref{thm-c}, we demand the following lemma.

\begin{lem}\label{lemm-c}
The series
\begin{align*}
g(c)=\sum_{k=0}^{\infty}\frac{2-c+2k}{(1+2k)^2}\frac{(\frac{3}{2}-c)_k^2}{(\frac{5}{2}-c)_{k}^2}
\end{align*}
is uniformly convergent in the open interval
$(\frac{3}{4},\frac{5}{4})$.
\end{lem}

\begin{proof}
For any $c\in(\frac{3}{4},\frac{5}{4})$, there holds
\begin{align*}
g(c)=\sum_{k=0}^{\infty}\frac{2-c+2k}{(1+2k)^2}\frac{(3-2c)^2}{(3-2c+2k)^2}
\leq\sum_{k=0}^{\infty}\frac{2+4k}{(1+2k)^2}\frac{(\frac{3}{2})^2}{(\frac{1}{2}+k)^2}
=\sum_{k=0}^{\infty}\frac{18}{(1+2k)^3}.
\end{align*}
Noting that the right series is convergent, we infer that $g(c)$ is
uniformly convergent in the interval $(\frac{3}{4},\frac{5}{4})$ via
Weierstrass's M-test.
\end{proof}

\begin{proof}[{\bf{Proof of Theorem \ref{thm-c}}}]
Memorize a hypergeometric transformation (cf. \cite[Theorem
9]{Chu-b}):
\begin{align}
&\sum_{k=0}^{\infty}\frac{(c)_k(d)_k(e)_k(1+a-b-c)_k(1+a-b-d)_{k}(1+a-b-e)_{k}}{(1+a-c)_{k}(1+a-d)_{k}(1+a-e)_{k}(1+2a-b-c-d-e)_{k}}
\notag\\[1mm]
&\quad\times\frac{(-1)^k}{(1+a-b)_{2k}}\beta_k(a,b,c,d,e)
\notag\\[1mm]
&\:=\sum_{k=0}^{\infty}(a+2k)\frac{(b)_k(c)_k(d)_k(e)_k}{(1+a-b)_{k}(1+a-c)_{k}(1+a-d)_{k}(1+a-e)_{k}},
\label{equation-bb}
\end{align}
where  $\mathfrak{R}(1+2a-b-c-d-e)>0$ and
\begin{align*}
\beta_k(a,b,c,d,e)&=\frac{(1+2a-b-c-d+2k)(a-e+k)}{1+2a-b-c-d-e+k}
\\[1mm]
&\quad+\frac{(1+a-b-c+k)(1+a-b-d+k)(e+k)}{(1+a-b+2k)(1+2a-b-c-d-e+k)}.
\end{align*}
Set $(a,b,d,e)=(\frac{3}{2},\frac{1}{2},1,1)$ in \eqref{equation-bb}
to obtain
\begin{align}
&\sum_{k=0}^{\infty}\bigg(\frac{-1}{4}\bigg)^k\frac{(c)_k(2-c)_k(1)_k^3}{(\frac{3}{2}-c)_{k}(\frac{5}{2}-c)_{k}(\frac{3}{2})_{k}^3}
\frac{9-4c+20k-6ck+10k^2}{3-2c+2k}
\notag\\[1mm]
&\:=\sum_{k=0}^{\infty}\frac{3+4k}{(1+k)(1+2k)}\frac{(c)_k(1)_k}{(\frac{5}{2}-c)_{k}(\frac{3}{2})_{k}}.
\label{eq:wei-gg}
\end{align}
The $(a,b,d,e,f,n)\to(\frac{3}{2},1,1,\frac{1}{2},1,\infty)$ case of
\eqref{9F8} can be written as
\begin{align}
\sum_{k=0}^{\infty}\frac{3+4k}{(1+k)(1+2k)}\frac{(c)_k(1)_k}{(\frac{5}{2}-c)_{k}(\frac{3}{2})_{k}}
=\frac{4}{3-2c}\sum_{k=0}^{\infty}\frac{2-c+2k}{(1+2k)^2}\frac{(\frac{3}{2}-c)_k^2}{(\frac{5}{2}-c)_{k}^2}.
\label{eq:wei-hh}
\end{align}
Substituting \eqref{eq:wei-hh} into \eqref{eq:wei-gg}, we get
\begin{align}
&\sum_{k=0}^{\infty}\bigg(\frac{-1}{4}\bigg)^k\frac{(c)_k(2-c)_k(1)_k^3}{(\frac{3}{2}-c)_{k}(\frac{5}{2}-c)_{k}(\frac{3}{2})_{k}^3}
\frac{9-4c+20k-6ck+10k^2}{3-2c+2k}
\notag\\[1mm]
&\:=\sum_{k=0}^{\infty}\frac{4(2-c+2k)}{(1+2k)^2(3-2c+2k)}\frac{(\frac{3}{2}-c)_k}{(\frac{5}{2}-c)_{k}}.
\label{eq:wei-ii}
\end{align}
Lemma \ref{lemm-b} shows that the right series is uniformly
convergent in $(\frac{3}{4},\frac{5}{4})$. It is evident that the
left series  is uniformly convergent in the same interval. Apply the
operator $\mathcal{D}_{c}$ on both sides of \eqref{eq:wei-ii} to
gain

\begin{align*}
&\sum_{k=0}^{\infty}\bigg(\frac{-1}{4}\bigg)^k\frac{(c)_k(2-c)_k(1)_k^3}{(\frac{3}{2}-c)_{k}(\frac{5}{2}-c)_{k}(\frac{3}{2})_{k}^3}
\frac{9-4c+20k-6ck+10k^2}{3-2c+2k}
\notag\\[1mm]
&\quad\times\Big\{H_{k}(c-1)-H_{k}(1-c)+H_{k}(\tfrac{1}{2}-c)+H_{k}(\tfrac{3}{2}-c)\Big\}
\notag\\[1mm]
&\:+\sum_{k=0}^{\infty}\bigg(\frac{-1}{4}\bigg)^k\frac{(c)_k(2-c)_k(1)_k^3}{(\frac{3}{2}-c)_{k}(\frac{5}{2}-c)_{k}(\frac{3}{2})_{k}^3}
\frac{2(1+k)(3+4k)}{(3-2c+2k)^2}
\notag\\[1mm]
&\:=\sum_{k=0}^{\infty}\frac{4(2-c+2k)}{(1+2k)^2(3-2c+2k)}\frac{(\frac{3}{2}-c)_k}{(\frac{5}{2}-c)_{k}}
\Big\{H_{k}(\tfrac{3}{2}-c)-H_{k}(\tfrac{1}{2}-c)\Big\}
\notag\\[1mm]
&\:+\sum_{k=0}^{\infty}\frac{4}{(1+2k)(3-2c+2k)^2}\frac{(\frac{3}{2}-c)_k}{(\frac{5}{2}-c)_{k}}.
\end{align*}
The $c\to1$ case of  the last equation reads
\begin{align}
&\sum_{k=0}^{\infty}\bigg(\frac{-1}{4}\bigg)^k\frac{(1)_k^5}{(\frac{3}{2})_k^5}
\Big\{(5+14k+10k^2)\Big(2H_{1+2k}-H_{k}-1\Big)-3k-2\Big\}
\notag\\[1mm]
&\:=\frac{45}{8}\zeta(4)-\frac{7}{2}\zeta(3).
 \label{eq:wei-jj}
\end{align}
By the aid of \eqref{zeta-tb} and \eqref{eq:wei-jj}, we may catch
hold of \eqref{eq:wei-c}.
\end{proof}

 For the sake of proving Theorem \ref{thm-d}, we need the
following lemma.

\begin{lem}\label{lemm-d}
The series
\begin{align*}
g(c)=\sum_{k=0}^{\infty}\frac{1}{1+2k}\frac{(\frac{3}{2}-c)_k(c-\frac{1}{2})_k}{(\frac{1}{2}+c)_k(\frac{5}{2}-c)_k}
\end{align*}
is uniformly convergent in the open interval
$(\frac{3}{4},\frac{5}{4})$.
\end{lem}

\begin{proof}
Let
\begin{align*}
v_k(c)=\frac{1}{1+2k},\quad
t_{n}(c)=\sum_{k=0}^n\frac{(\frac{3}{2}-c)_k(c-\frac{1}{2})_k}{(\frac{1}{2}+c)_k(\frac{5}{2}-c)_k}.
\end{align*}
The sequence $\{v_k(c)\}_{k\geq0}$ is descending on $k$ and is
uniformly convergent to $0$ in $(\frac{3}{4},\frac{5}{4})$. In view
 of the relation:
\begin{align*}
t_{n}(c)&=\sum_{k=0}^n\frac{(\frac{3}{2}-c)(c-\frac{1}{2})}{(\frac{3}{2}-c+k)(c-\frac{1}{2}+k)}
\leq\frac{9}{16}\sum_{k=0}^n\frac{1}{(\frac{1}{4}+k)^2}\leq9\sum_{k=0}^{\infty}\frac{1}{(1+4k)^2}
\\
&\leq9\sum_{k=0}^{\infty}\frac{1}{(1+2k)^2}=\frac{9}{8}\pi^2,
\end{align*}
we judge that the sequence $\{t_n(c)\}_{n\geq0}$ is
  uniformly bounded in $(\frac{3}{4},\frac{5}{4})$. Through
Dirichlet's uniform convergence test, we think that $g(c)$ is
uniformly convergent in $(\frac{3}{4},\frac{5}{4})$.
 \end{proof}

\begin{proof}[{\bf{Proof of Theorem \ref{thm-d}}}]
Select $(a,b,d,e)=(\frac{3}{2},\frac{1}{2},2-c,1)$ in
\eqref{equation-bb} to procure
\begin{align}
&\sum_{k=0}^{\infty}\bigg(\frac{-1}{4}\bigg)^k\frac{(c)_k^2(2-c)_k^2(1)_k}{(\frac{5}{2}-c)_{k}(\frac{1}{2}+c)_{k}(\frac{1}{2})_{k}(\frac{3}{2})_{k}^2}
\bigg\{(3+4k)+\frac{2(c+k)(2-c+k)}{1+2k}\bigg\}
\notag\\[1mm]
&\:=\sum_{k=0}^{\infty}\frac{3+4k}{(1+k)(1+2k)}\frac{(c)_k(2-c)_k}{(\frac{5}{2}-c)_{k}(\frac{1}{2}+c)_{k}}.
\label{eq:wei-kk}
\end{align}
The $(a,b,d,e,f,n)\to(\frac{3}{2},1,2-c,\frac{1}{2},1,\infty)$ case
of  \eqref{9F8} can be stated as
\begin{align}
\sum_{k=0}^{\infty}\frac{3+4k}{(1+k)(1+2k)}\frac{(c)_k(2-c)_k}{(\frac{5}{2}-c)_{k}(\frac{1}{2}+c)_{k}}
=\sum_{k=0}^{\infty}\frac{4}{1+2k}\frac{(\frac{3}{2}-c)_k(c-\frac{1}{2})_k}{(\frac{1}{2}+c)_k(\frac{5}{2}-c)_k}.
\label{eq:wei-lla}
\end{align}
Substituting \eqref{eq:wei-lla} into \eqref{eq:wei-kk}, we have
\begin{align}
&\sum_{k=0}^{\infty}\bigg(\frac{-1}{4}\bigg)^k\frac{(c)_k^2(2-c)_k^2(1)_k}{(\frac{5}{2}-c)_{k}(\frac{1}{2}+c)_{k}(\frac{1}{2})_{k}(\frac{3}{2})_{k}^2}
\bigg\{(3+4k)+\frac{2(c+k)(2-c+k)}{1+2k}\bigg\}
\notag\\
&\:=\sum_{k=0}^{\infty}\frac{4}{1+2k}\frac{(\frac{3}{2}-c)_k(c-\frac{1}{2})_k}{(\frac{1}{2}+c)_k(\frac{5}{2}-c)_k}.
\label{eq:wei-ll}
\end{align}
Lemma \ref{lemm-d} demonstrates that the right series is uniformly
convergent in $(\frac{3}{4},\frac{5}{4})$. It is apparent that the
left series is uniformly convergent in the same interval. Employ the
operator $\mathcal{D}_{c}$ on both sides of \eqref{eq:wei-ll} to
find
\begin{align*}
&\sum_{k=0}^{\infty}\bigg(\frac{-1}{4}\bigg)^k\frac{(c)_k^2(2-c)_k^2(1)_k}{(\frac{5}{2}-c)_{k}(\frac{1}{2}+c)_{k}(\frac{1}{2})_{k}(\frac{3}{2})_{k}^2}
\bigg\{(3+4k)+\frac{2(c+k)(2-c+k)}{1+2k}\bigg\}
\notag\\[1mm]
&\quad\times\Big\{2H_{k}(c-1)-2H_{k}(1-c)+H_{k}(\tfrac{3}{2}-c)-H_{k}(c-\tfrac{1}{2})\Big\}
\notag\\
&\:+\sum_{k=0}^{\infty}\bigg(\frac{-1}{4}\bigg)^k\frac{(c)_k^2(2-c)_k^2(1)_k}{(\frac{5}{2}-c)_{k}(\frac{1}{2}+c)_{k}(\frac{1}{2})_{k}(\frac{3}{2})_{k}^2}
\frac{4(1-c)}{1+2k}
\notag\\
&\:=\sum_{k=0}^{\infty}\frac{4}{1+2k}\frac{(\frac{3}{2}-c)_k(c-\frac{1}{2})_k}{(\frac{1}{2}+c)_k(\frac{5}{2}-c)_k}
\Big\{H_{k}(c-\tfrac{3}{2})-H_{k}(\tfrac{1}{2}-c)+H_{k}(\tfrac{3}{2}-c)-H_{k}(c-\tfrac{1}{2})\Big\}.
\end{align*}
Dividing both sides by $2-2c$ and then letting $c\to1$, it is
routine to see that
\begin{align}
&\sum_{k=0}^{\infty}\bigg(\frac{-1}{4}\bigg)^k\frac{(1)_k^5}{(\frac{3}{2})_k^5}
\Big\{(5+14k+10k^2)\Big[4H_{1+2k}^{(2)}-3H_{k}^{(2)}-4\Big]-2\Big\}
\notag\\
&\:=\frac{31}{2}\zeta(5)-14\zeta(3).
 \label{eq:wei-mm}
\end{align}
With the help of \eqref{zeta-tb} and \eqref{eq:wei-mm}, we are led
to \eqref{eq:wei-d}.
\end{proof}
\section{Proof of Theorems
\ref{thm-e}-\ref{thm-h}}

In this section,  we shall use Lemmas \ref{lemm-a}-\ref{lemm-b}
without explanation. For proving Theorems \ref{thm-e}-\ref{thm-h},
we require the following hypergeometric transformation (cf.
\cite[Theorem 14]{Chu-b}):
\begin{align}
&\sum_{k=0}^{\infty}\frac{(c)_k(e)_k(1+a-b-c)_k(1+a-b-e)_{k}(1+a-c-d)_{k}(1+a-d-e)_{k}}{(1+a-c)_{k}(1+a-e)_{k}}
\notag\\[1mm]
&\quad\times\frac{(1+a-b-d)_{2k}\,\omega_k(a,b,c,d,e)}{(1+a-b)_{2k}(1+a-d)_{2k}(1+2a-b-c-d-e)_{2k}}
\notag\\[1mm]
&=\sum_{k=0}^{\infty}(a+2k)\frac{(b)_k(c)_k(d)_k(e)_k}{(1+a-b)_{k}(1+a-c)_{k}(1+a-d)_{k}(1+a-e)_{k}},
\label{equation-th}
\end{align}
where  $\mathfrak{R}(1+2a-b-c-d-e)>0$ and
\begin{align*}
\omega_k(a,b,c,d,e)&=\frac{(1+2a-b-c-d+3k)(a-e+k)}{1+2a-b-c-d-e+2k}
\\[1mm]
&\quad+\frac{(e+k)(1+a-b-c+k)}{(1+a-b+2k)(1+a-d+2k)}
\\[1mm]
&\quad\times\frac{(1+a-c-d+k)(1+a-b-d+2k)(2+2a-b-d-e+3k)}{(1+2a-b-c-d-e+2k)(2+2a-b-c-d-e+2k)}.
\end{align*}

Firstly, we plan to prove Theorem \ref{thm-e}.

\begin{proof}[{\bf{Proof of Theorem \ref{thm-e}}}]
Choose $(a,b,d,e)=(2,1,1,1)$ in \eqref{equation-th} to obtain
\begin{align}
\sum_{k=0}^{\infty}\bigg(\frac{1}{4}\bigg)^k\frac{(1)_k(c)_k(2-c)_k^2}{(\frac{3}{2})_{k}(3-c)_{k}(2-c)_{2k}}
E_k(c)=\sum_{k=0}^{\infty}\frac{2}{(1+k)^2}\frac{(c)_k}{(3-c)_{k}}.
\label{eq:wei-tha}
\end{align}
where
\begin{align*}
E_k(c)=
\frac{3-c+3k}{(1+2k)(2-c+2k)}+\frac{3(2-c+k)^2}{4(1+k)(2-c+2k)(3-c+2k)}.
\end{align*}
It is not difficult to realize that these series on both sides are
uniformly convergent in $(\frac{1}{2},\frac{3}{2})$. Apply the
operator $\mathcal{D}_{c}$ on both sides of \eqref{eq:wei-tha} to
get
\begin{align*}
&\sum_{k=0}^{\infty}\bigg(\frac{1}{4}\bigg)^k\frac{(1)_k(c)_k(2-c)_k^2}{(\frac{3}{2})_{k}(3-c)_{k}(2-c)_{2k}}E_k(c)
\notag\\[1mm]
&\quad\times\Big\{H_{k}(c-1)+H_{k}(2-c)+H_{2k}(1-c)-2H_{k}(1-c)\Big\}
\notag\\[1mm]
&\:+\sum_{k=0}^{\infty}\bigg(\frac{1}{4}\bigg)^k\frac{(1)_k(c)_k(2-c)_k^2}{(\frac{3}{2})_{k}(3-c)_{k}(2-c)_{2k}}\mathcal{D}_{c}E_k(c)
\notag\\[1mm]
&\:\:= \sum_{k=0}^{\infty}\frac{2}{(1+k)^2}\frac{(c)_k}{(3-c)_{k}}
\Big\{H_{k}(c-1)+H_{k}(2-c)\Big\}.
\end{align*}
The $b\to1$ case of it becomes
\begin{align}
&\sum_{k=0}^{\infty}\bigg(\frac{1}{16}\bigg)^k\frac{(1)_k^2}{(\frac{3}{2})_k^2}
\bigg\{\frac{19+30k}{(1+k)(1+2k)}\Big(H_{1+2k}-1\Big)+\frac{13}{2(1+k)^2}\bigg\}
\notag\\[1mm]
&\:=32\sum_{k=0}^{\infty}\frac{H_{1+k}}{(1+k)^3}-16\zeta(4)-16\zeta(3).
\label{eq:wei-thb}
\end{align}
In terms of \eqref{zeta-td}, \eqref{zeta4-a}, and
\eqref{eq:wei-thb}, we can arrive at \eqref{eq:wei-e}.
\end{proof}

Secondly, we start to prove Theorem \ref{thm-f}.

\begin{proof}[{\bf{Proof of Theorem \ref{thm-f}}}]
Fix $(a,c,d,e)=(2,1,1,1)$ in \eqref{equation-th} to deduce
\begin{align}
\sum_{k=0}^{\infty}\bigg(\frac{1}{4}\bigg)^k\frac{(1)_k(2-b)_k^2}{(\frac{3}{2})_{k}(3-b)_{2k}}
F_k(b)=\sum_{k=0}^{\infty}\frac{2}{(1+k)^2}\frac{(b)_k}{(3-b)_{k}}.
\label{eq:wei-thc}
\end{align}
where
\begin{align*}
F_k(b)=
\frac{3-b+3k}{(1+k)(2-b+2k)}+\frac{(2-b+k)(4-b+3k)}{2(1+k)(3-b+2k)^2}.
\end{align*}
It is clear that these series on both sides are uniformly convergent
in $(\frac{1}{2},\frac{3}{2})$. Employ the operator
$\mathcal{D}_{b}$ on both sides of \eqref{eq:wei-thc} to find
\begin{align*}
&\sum_{k=0}^{\infty}\bigg(\frac{1}{4}\bigg)^k\frac{(1)_k(2-b)_k^2}{(\frac{3}{2})_{k}(3-b)_{2k}}
F_k(b)\Big\{H_{2k}(2-b)-2H_{k}(1-b)\Big\}
\notag\\[1mm]
&\:+\sum_{k=0}^{\infty}\bigg(\frac{1}{4}\bigg)^k\frac{(1)_k(2-b)_k^2}{(\frac{3}{2})_{k}(3-b)_{2k}}\mathcal{D}_{b}F_k(b)
\notag\\[1mm]
&\:\:= \sum_{k=0}^{\infty}\frac{2}{(1+k)^2}\frac{(b)_k}{(3-b)_{k}}
\Big\{H_{k}(b-1)+H_{k}(2-b)\Big\}.
\end{align*}
The $b\to1$ case of the last identity is
\begin{align}
&\sum_{k=0}^{\infty}\bigg(\frac{1}{16}\bigg)^k\frac{(1)_k^2}{(\frac{3}{2})_k^2}
\bigg\{\frac{19+30k}{(1+k)(1+2k)}\Big(H_{1+2k}-2H_{k}-1\Big)+\frac{8}{(1+2k)^2}-\frac{1}{(1+k)^2}\bigg\}
\notag\\[1mm]
&\:=32\sum_{k=0}^{\infty}\frac{H_{1+k}}{(1+k)^3}-16\zeta(4)-16\zeta(3).
\label{eq:wei-thd}
\end{align}
By means of \eqref{zeta-td}, \eqref{zeta4-a}, \eqref{eq:wei-e}, and
\eqref{eq:wei-thd}, we can catch hold of \eqref{eq:wei-f}.
\end{proof}

Thirdly, we begin to prove Theorem \ref{thm-g}.

\begin{proof}
Set $(a,b,d,e)=(2,1,1,2-c)$ in \eqref{equation-th} to obtain
\begin{align}
&\sum_{k=0}^{\infty}\frac{(c)_k^3(2-c)_k^3}{(3-c)_{k}(1+c)_{k}(2)_{2k}^2}
\bigg\{\frac{(c+k)(3-c+3k)}{1+2k}+\frac{(2+c+3k)(2-c+k)^3}{8(1+k)^3}\bigg\}
\notag\\
&\:=\sum_{k=0}^{\infty}\frac{2}{1+k}\frac{(c)_k(2-c)_k}{(3-c)_{k}(1+c)_{k}}.
\label{eq:wei-the}
\end{align}
It is easy to understand that these series on both sides are
uniformly convergent in $(\frac{1}{2},\frac{3}{2})$. Apply the
operator $\mathcal{D}_{c}$ on both sides of \eqref{eq:wei-the} to
get
\begin{align*}
&\sum_{k=0}^{\infty}\frac{(c)_k^3(2-c)_k^3}{(3-c)_{k}(1+c)_{k}(2)_{2k}^2}
\bigg\{\frac{(c+k)(3-c+3k)}{1+2k}+\frac{(2+c+3k)(2-c+k)^3}{8(1+k)^3}\bigg\}
\notag\\[1mm]
&\quad\times\Big\{3H_{k}(c-1)-3H_{k}(1-c)+H_{k}(2-c)-H_{k}(c)\Big\}
\notag\\[1mm]
&\:+\sum_{k=0}^{\infty}\frac{(c)_k^3(2-c)_k^3}{(3-c)_{k}(1+c)_{k}(2)_{2k}^2}
\frac{(1-c)(2-2c+c^2+2k-4ck+2c^2k-3k^2-2k^3)}{2(1+k)^3(1+2k)}
\notag\\[1mm]
&\:\:=
\sum_{k=0}^{\infty}\frac{2}{1+k}\frac{(c)_k(2-c)_k}{(3-c)_{k}(1+c)_{k}}
\Big\{H_{k}(c-1)-H_{k}(1-c)+H_{k}(2-c)-H_{k}(c)\Big\}.
\end{align*}
Dividing both sides by $2-2c$ and then taking $c\to1$, we have
\begin{align}
&\sum_{k=0}^{\infty}\bigg(\frac{1}{16}\bigg)^k\frac{(1)_k^2}{(\frac{3}{2})_k^2}
\bigg\{\frac{19+30k}{(1+k)(1+2k)}\Big[2H_{k}^{(2)}+1\Big]-\frac{17}{(1+k)^3}\bigg\}
\notag\\[1mm]
&\:=16\zeta(3)-16\zeta(5).
 \label{eq:wei-thf}
\end{align}
According to \eqref{zeta-td} and \eqref{eq:wei-thf}, we discover
\eqref{eq:wei-g}.
\end{proof}

Finally, we shall certify Theorem \ref{thm-h}.

\begin{proof}
Select $(a,c,d,e)=(2,1,2-b,1)$ in \eqref{equation-th} to procure
\begin{align}
&\sum_{k=0}^{\infty}\frac{(b)_k^2(2-b)_k^2}{(3-b)_{2k}(1+b)_{2k}}
\bigg\{\frac{2+3k}{(1+k)(1+2k)}+\frac{3(b+k)(2-b+k)}{2(1+k)(1+b+2k)(3-b+2k)}\bigg\}
\notag\\
&\:=\sum_{k=0}^{\infty}\frac{2}{1+k}\frac{(b)_k(2-b)_k}{(3-b)_{k}(1+b)_{k}}.
\label{eq:wei-tho}
\end{align}
It is evident that these series on both sides are uniformly
convergent in $(\frac{1}{2},\frac{3}{2})$. Apply the operator
$\mathcal{D}_{b}$ on both sides of \eqref{eq:wei-tho} to gain

\begin{align*}
&\sum_{k=0}^{\infty}\frac{(b)_k^2(2-b)_k^2}{(3-b)_{2k}(1+b)_{2k}}
\bigg\{\frac{2+3k}{(1+k)(1+2k)}+\frac{3(b+k)(2-b+k)}{2(1+k)(1+b+2k)(3-b+2k)}\bigg\}
\notag\\[2mm]
&\quad\times\Big\{2H_{k}(b-1)-2H_{k}(1-b)+H_{2k}(2-b)-H_{2k}(b)\Big\}
\notag\\[1mm]
&\:+\sum_{k=0}^{\infty}\frac{(b)_k^2(2-b)_k^2}{(3-b)_{2k}(1+b)_{2k}}
\frac{9(1+k)(1-b)}{(1+b+2k)^2(3-b+2k)^2}
\notag\\[1mm]
&\:\:=
\sum_{k=0}^{\infty}\frac{2}{1+k}\frac{(b)_k(2-b)_k}{(3-b)_{k}(1+b)_{k}}
\Big\{H_{k}(b-1)-H_{k}(1-b)+H_{k}(2-b)-H_{k}(b)\Big\}.
\end{align*}
Dividing both sides by $2-2b$ and then taking $b\to1$, there holds
\begin{align}
&\sum_{k=0}^{\infty}\bigg(\frac{1}{16}\bigg)^k\frac{(1)_k^2}{(\frac{3}{2})_k^2}
\bigg\{\frac{19+30k}{(1+k)(1+2k)}\Big[H_{1+2k}^{(2)}-2H_{k}^{(2)}-1\Big]-\frac{9}{4(1+k)^3}\bigg\}
\notag\\[1mm]
&\:=16\zeta(5)-16\zeta(3).
 \label{eq:wei-thp}
\end{align}
With the help of \eqref{zeta-td} and \eqref{eq:wei-thp}, we are led
to \eqref{eq:wei-h}.
\end{proof}



\begin{thebibliography}{99}
\small \setlength{\itemsep}{-.8mm}



\bibitem{Amdeberhan-b} T. Amdeberhan, D. Zeilberger, Hypergeometric series acceleration
via the WZ method, Electron. J. Combin. 4 (2) (1997), \#R3.

\bibitem{Au} K.C. Au, Colored multiple zeta values, WZ-pairs and some infinite sums, preprint, arXiv: 2212. 02986v2.

\bibitem{Bailey} W.N. Bailey, Generalized Hypergeometric Series, Cambridge University Press, Cambridge,
1935.

\bibitem{Chu-b} W. Chu, W. Zhang, Accelerating Dougall's $_5F_4$-sum and infinite series involving $\pi$, Math
Comput. 285 (2014), 475--512.

\bibitem{Schlosser}A. Berkovich, H.H. Chan, M.J. Schlosser, Wronskians of theta functions and series for $1/\pi$, Adv. Math. 338 (2018), 266--304.


\bibitem{Borwein} J.M. Borwein, D.M. Bradley, Empirically determined Ap\'{e}ry-like formulae for $\zeta(4n+3)$, Exp. Math. 6 (1997), 181--194.

\bibitem{Chan}H.H. Chan, S.H. Chan, Z. Liu, Domb's numbers and Ramanujan--Sato
type series for $1/\pi$, Adv. Math. 186 (2004), 396--410.


\bibitem{Guillera} J. Guillera, Hypergeometric identities for 10 extended
Ramanujan-type series, Ramanujan J. 15 (2008), 219--234.

\bibitem{Guo}V.J.W. Guo, W. Zudilin, Ramanujan-type formulae for $1/\pi$: $q$-analogues, Integral Transforms Spec. Funct. 29 (2018), 505--513.

\bibitem{Sun}Z.-W. Sun, Series with summands involving harmonic numbers, preprint, arXiv: 2210. 07238v7.

\bibitem{Wang} L. Wang, Y. Yang, Ramanujan-type $1/\pi$-series from bimodular
forms, Ramanujan J. (2022).
https://doi.org/10.1007/s11139-021-00532-6.

\bibitem{Weisstein} J. Sondow, W.E. Weisstein, Harmonic Number. From
MathWorld--A Wolfram Web Resource.
https://mathworld.wolfram.com/HarmonicNumber.html.

\end{thebibliography}
\end{document}